\documentclass[journal,10pt,twocolumn,oneside, peereview]{IEEEtran}
\usepackage{enumerate}
\usepackage{float}
\usepackage{amssymb,amsmath,amsxtra,amsfonts,amsthm,mathtools}
\usepackage{graphicx}
\usepackage{bigints}
\usepackage{fancyhdr}
\usepackage{float}
\usepackage{times}
\usepackage{wrapfig}
\usepackage{xcolor,colortbl}
\usepackage{soul}
\usepackage{array,multirow,makecell}
\usepackage{subfig}
\usepackage{nopageno}
\usepackage{tikz}
\usepackage{glossaries}
\usepackage{hyperref}
\usepackage[switch,columnwise]{lineno} 
\usepackage{mathrsfs}
\usepackage{lineno}
\usepackage[
backend=biber,
style=numeric,
sorting=ynt
]{biblatex}

\usepackage{hyperref}
\usepackage{float}
\usepackage{enumitem}

\usepackage{algorithm}
\usepackage{algpseudocode}
\usepackage{amsthm}

\usepackage{amsthm}
\newtheorem{definition}{Definition}

\newtheorem{lemma}{Lemma}

\newtheorem{theorem}{Theorem}

\newtheorem{assumption}{Assumption}

\begin{document}
\title{\textbf{A Low-Power Hardware-Friendly Optimisation Algorithm With Absolute Numerical Stability and Convergence Guarantees}}

\author{
  Anis Hamadouche,\,\,
  Yun Wu,\,\,
  Andrew M.\ Wallace,\,
  and\,
  Jo\~ao F.\ C.\ Mota%
  \IEEEcompsocitemizethanks{
    \IEEEcompsocthanksitem
    Work supported by UK's EPSRC (EP/T026111/1, EP/S000631/1), and the MOD University Defence
    Research Collaboration. 
    \IEEEcompsocthanksitem
    Anis Hamadouche, Andrew M.\ Wallace, and Jo\~ao F.\ C.\ Mota
    are with the School of Engineering \& Physical Sciences, Heriot-Watt University, Edinburgh EH14 4AS,
    UK. (e-mail: \{ah225,y.wu,a.m.wallace,j.mota\}@hw.ac.uk). Yun Wu is with the School of Electronics, Electrical Engineering and Computer Science, Queen's University Belfast, Belfast BT7 1NN, Ireland. (e-mail: ywu12@qub.ac.uk).
  }
}

\maketitle
\thispagestyle{plain}
\pagestyle{plain}

\begin{abstract}
\textbf{We propose Dual-Feedback Generalized Proximal Gradient Descent (DFGPGD) as a new, hardware-friendly, operator splitting algorithm. We then establish convergence guarantees under approximate computational errors and we derive theoretical criteria for the numerical stability of DFGPGD based on absolute stability of dynamical systems. We also propose a new generalized proximal ADMM that can be used to instantiate most of existing proximal-based composite optimization solvers. We implement DFGPGD and ADMM on FPGA ZCU106 board and compare them in light of FPGA's timing as well as resource utilization and power efficiency. We also perform a full-stack, application-to-hardware, comparison between approximate versions of DFGPGD and ADMM based on dynamic power/error rate trade-off, which is a new hardware-application combined metric.}
\end{abstract}
\begin{IEEEkeywords}
\textbf{Optimization; Nonlinear Programming; DNLP; Approximate Computing; ADMM; Proximal Gradient Method; FPGA; Low-power.}
\end{IEEEkeywords}

\IEEEpeerreviewmaketitle
\newcommand\scalemath[2]{\scalebox{#1}{\mbox{\ensuremath{\displaystyle #2}}}}
\section{INTRODUCTION}
Consider a constrained penalized quadratic optimization problem 
\begin{equation}
 \label{Eq:LassoProblem0}
  \begin{split}
      &\underset{x \in \mathbb{R}^n,z \in \mathbb{R}^n}{\text{minimize}} \,\,\,
      f(x) := \|Hx-b\|_2^2 + h(z),\\
      &\text{subject to} \,\,\,
      Ax+Bz = c\,.  
  \end{split}
\end{equation}
where $H \in \mathbb{R}^{m\times n}$, $b \in \mathbb{R}^{m}$, $c \in \mathbb{R}^{p}$, $A\in \mathbb{R}^{p\times n}$ and $B \in \mathbb{R}^{p\times m}$ are fixed. 
This problem can be solved using a class of first-order algorithms known as the Alternating Direction Method of Multipliers (ADMM), which we introduce next.
\subsection{Weighted Lagrangian ADMM (WL-ADMM)}
The generalized WL-ADMM is formulated by taking the $L_k$-weighted Lagrangian of~\eqref{Eq:LassoProblem0}, where $L_k$ is a positive definite (PD) matrix:
\begin{subequations}
\label{WL-ADMM}
    \begin{alignat}{3}
    \label{WLADMM1}
    &x^{k+1} = \underset{x}{\arg\min}\quad g(x)+\frac{1}{2\lambda}\|Ax+Bz^k-c+v^k\|_{L_k }^2&\\
    \label{WLADMM2}
    &z^{k+1} = \underset{z}{\arg\min}\quad h(z)+\frac{1}{2\lambda}\|Ax^{k+1}+Bz-c+v^k\|_{L_k }^2&\\
    \label{WLADMM3}
    &v^{k+1} = v^k+(Ax^{k+1}+Bz^{k+1}-c).&
    \end{alignat}
\end{subequations}
We recover the classical ADMM scheme~\cite{boyd2011distributed} when $L_k$ is the identity matrix.
\subsection{Weighted Lagrangian M-Generalized Proximal ADMM (WLM-ADMM)}
The \textit{Weighted Lagrangian M-Generalized proximal ADMM} is obtained by adding a quadratic function in \eqref{WLADMM1} and \eqref{WLADMM2} penalizing the difference between the variable and its previous value, a term known as \textit{proximal penalty}. This yields
\begin{subequations}
\label{WLM-ADMM}
    \begin{alignat}{3}
    &x^{k+1} = \underset{x}{\arg\min}\quad g(x)+\frac{1}{2\lambda}\|Ax+Bz^k-c+v^k\|_{L_k }^2&\notag\\&+\frac{1}{2}\|x-x^k\|_{M_x^k}^2&\\
    &z^{k+1} = \underset{z}{\arg\min}\quad h(z)+\frac{1}{2\lambda}\|Ax^{k+1}+Bz-c+v^k\|_{L_k }^2&\notag\\&+\frac{1}{2}\|z-z^k\|_{M_z^k}^2&\\
    &v^{k+1} = v^k+(Ax^{k+1}+Bz^{k+1}-c),&
    \end{alignat}
\end{subequations}
where $M_x^k$ and $M_z^k$ are PD matrices that can vary from iteration to iteration. We recover~\eqref{WL-ADMM} when both matrices are set to zero. Note that in the case of convex function $f$, adding quadratic terms (proximal terms) makes the objective functions strongly convex, which improves the condition number of each subproblem, at the expense of yielding approximate solutions.  

Defining
\begin{subequations}
\label{WLM-ADMM1.2Vars}
    \begin{alignat}{3}
    & \Lambda_{1_k} = \frac{1}{\lambda} A^\top L_k  A + M_x^k&\\
    & \gamma_{1_k}(x^k,z^k,v^k) = M_x^k x^k-\frac{1}{\lambda}A^\top L_k  (Bz^k-c+v^k)&\\
    & \Lambda_{2_k} = \frac{1}{\lambda} B^\top L_k  B + M_z^k&\\
    & \gamma_{2_k}(x^{k+1},z^k,v^k) = M_z^k z^k-\frac{1}{\lambda}B^\top L_k  (Ax^{k+1}-c+v^k).&   
    \end{alignat}
\end{subequations}
we can write \eqref{WLM-ADMM} using Euclidean norms, as
\begin{subequations}
\label{WLM-ADMM1.1}
    \begin{alignat}{3}
    &x^{k+1} = \underset{x}{\arg\min}\quad g(x)+\frac{1}{2}\|\Lambda_{1_k}^{1/2}x -\Lambda_{1_k}^{-1/2}\gamma_{1_k}\|^2&\\
    &z^{k+1} = \underset{z}{\arg\min}\quad h(z)+\frac{1}{2}\|\Lambda_{2_k}^{1/2}z -\Lambda_{2_k}^{-1/2}\gamma_{2_k}\|^2&\\
    &v^{k+1} = v^k+(Ax^{k+1}+Bz^{k+1}-c).&
    \end{alignat}
\end{subequations}
Note that $\Lambda_{1_k}$ and $\Lambda_{2_k}$ are the sum of a positive semidefinite matrix and a positive definite one. Thus, they are always invertible for any A and B. In turn, \eqref{WLM-ADMM1.1} can be written with generalized norms, as
\begin{subequations}
\label{WLM-ADMMP1.2}
    \begin{alignat}{3}
    &x^{k+1} = \underset{x}{\arg\min}\quad g(x)+\frac{1}{2}\|x -\Lambda_{1_k}^{-1}\gamma_{1_k}\|_{\Lambda_{1_k}}^2&\\
    &z^{k+1} = \underset{z}{\arg\min}\quad h(z)+\frac{1}{2}\|z -\Lambda_{2_k}^{-1}\gamma_{2_k}\|_{\Lambda_{2_k}}^2&\\
    &v^{k+1} = v^k+(Ax^{k+1}+Bz^{k+1}-c).&
    \end{alignat}
\end{subequations}

In many hardware-based algorithms, matrix inversion can be very expensive and less stable in numerical calculation. Thus, in this work, we modify the WLM-ADMM algorithm to design a new algorithm that completely avoids the matrix inversion step. We propose Dual Feedback Generalized Proximal Gradient (DFGPGD) algorithm which is an instance of the Weighted-Lagrangian M-generalized-proximal ADMM (WLM-ADMM) but with a careful choice of the proximal matrix $M$.

\subsection{Motivation}
The WLM-ADMM instance for problem~\eqref{Eq:LassoProblem0} is given by

\begin{subequations}
    \begin{alignat}{3}
    &x^{k+1} =(H^\top H + \Lambda_{1_k})^{-1}(H^\top b+\gamma_{1_k})&\label{DFGPGD:x-update}\\
    &z^{k+1} = \underset{z}{\arg\min}\quad \gamma\|z\|_1+\frac{1}{2}\|z -\Lambda_{2_k}^{-1}\gamma_{2_k}\|_{\Lambda_{2_k}}^2&\\
    &v^{k+1} = v^k+(Ax^{k+1}+Bz^{k+1}-c).&
    \label{WLM-ADMM}
    \end{alignat}
\end{subequations}
where 
\begin{subequations}
    \begin{alignat}{3}
    & \Lambda_{1_k} = \frac{1}{\lambda} A^\top L  A + M_x&\\
    & \gamma_{1_k}(x^k,z^k,v^k) = \Lambda_{1_k} x^k-\frac{1}{\lambda}A^\top L  (Ax^k+Bz^k-c+v^k)&\\
    & \Lambda_{2_k} = \frac{1}{\lambda} B^\top L  B + M_z&\\
    & \gamma_{2_k}(x^{k+1},z^k,v^k) = \Lambda_{2_k} z^k-\frac{1}{\lambda}B^\top L  (Ax^{k+1}+Bz^k-c+v^k).&
    \end{alignat}
\end{subequations}
Note that $H^\top H + \Lambda_{1_k}$ is invertible since it is a sum of a positive semi-definite matrix, $H^\top H$, and a positive definite matrix,  $\Lambda_{1_k}$ (by assumption). Let us consider a scenario where the condition number, which is given by the ratio of the maximum and minimum eigenvalues 
\begin{equation}
    \rho = \frac{\lambda_{\max}(H^\top H + \Lambda_{1_k})}{\lambda_{\min}(H^\top H + \Lambda_{1_k})}
\end{equation}
  is very large so that the numerical computation of the inverse is difficult and expensive to obtain (which is typically the case in high-dimensional problems). We call such a case \textit{ill-conditioned}. 
 
\subsection{Our approach}  
For ill-conditioned problems, we can leverage the degrees of freedom of WLM-ADMM to design $L$ and/or $M_x$ to optimize for the numerical stability of the inverse in the $x$-update ~\eqref{DFGPGD:x-update}. For instance, if we choose $M_x$ as  
\begin{equation}
    M_x = \lambda_xI - H^\top H - \frac{1}{\lambda}A^\top L A,
\end{equation}
then we completely eliminate the inverse operation in~\eqref{DFGPGD:x-update} and we obtain the following, inverse-free, instance of the WLM-ADMM algorithm 
\begin{subequations}
\label{DFGPGD1}
    \begin{alignat}{3}
    & u^{k+1} = Ax^k+Bz^k-c+v^k,&\\
    & x^{k+1} = x^k-\frac{1}{\lambda_x}(H^\top H x^k -H^\top b)-\frac{1}{\lambda_x\lambda}A^\top L u^{k+1},&\label{DFGPGD1:x-update}\\
    & z^{k+1} = \underset{z}{\arg\min}\quad h(z)+\frac{1}{2}\|z -\Lambda_{2_k}^{-1}\gamma_{2_k}\|_{\Lambda_{2_k}}^2&\label{DFGPGD1:z-update}\\
    & v^{k+1} = v^k+(Ax^{k+1}+Bz^{k+1}-c).&\label{DFGPGD1:v-update}
    \end{alignat}
\end{subequations}
Although algorithm~\eqref{DFGPGD1} and ADMM share the same $z$ and $v$ updates, the $x$-update in~\eqref{DFGPGD1} is more computationally efficient. In contrast to ADMM, Algorithm~\eqref{DFGPGD1} has a $O(n^2)$ iteration complexity as it does not involve any matrix inversion step. Specifically, the $x$-update step in~\eqref{DFGPGD1} only requires matrix-matrix and matrix-vector basic addition and multiplications with a total of $6n + 4n^2+$ flops assuming $H^\top H$, $H^\top b$ and $A^\top L$ are pre-calculated. On the other hand, ADMM has a more computationally expensive $x$-update with $n^3+4n^2+5n$ flops, and therefore, has an iteration complexity of $O(n^3)$.
\subsection{Applications}
Many structured machine learning and AI problems can be formulated as composite optimization problems. First-order algorithms are best aligned with the constrained computational resources and timing requirements of edge applications. Hardware-friendly algorithms are in demand for more efficient data processing in battery-powered applications such as edge-AI, also referred to as "embedded intelligence", and resource-constrained applications that make use of hardware-based solvers such as ADMM and proximal gradient in computer vision, compressed sensing and control. In this paper, we use a synthetic LASSO problem for evaluation and comparison. LASSO is a paradigm that is typically encountered in a variety of real-world problems such as: sparse image reconstruction and model predictive control (MPC) with sparse control signals.   
\subsection{Contributions}
We propose and analyse the convergence and numerical stability of a new type of operator splitting algorithm, Dual-Feedback Generalized Proximal Gradient Descent (DFGPGD),  under computational inaccuracies due to inexact updates. We propose a dynamical system state space description for DFGPGD that we use to derive absolute numerical stability criteria. We implement DFGPGD and ADMM on FPGA ZCU106 board and perform a comparative study in light of FPGA's timing, resource utilization and power efficiency criteria. We propose a new combined hardware-application metric, i.e., "power/error rate trade-off" that we use to perform a comparative study between approximate versions of DFGPGD and ADMM.
\subsection{Organization}
In the following section, we introduce a general form of DFGPGD for composite convex problems. In Section~\ref{Section:Convergence analysis}, we establish convergence bounds for DFGPGD under approximate proximal computations. We also derive theoretical criteria for the absolute numerical stability of DFGPGD in the same section. In Section~\ref{Section:Experimental results}, we present our experimental results. We conclude the paper by a summary and possible future directions in Section~\ref{Section:Conclusions}

\section{Dual Feedback Generalized Proximal Gradient Descent (DFGPGD)}

Although by using Algorithm~\eqref{DFGPGD1} we completely avoid the computation of the inverse, this comes at the expense of an additional update step $u^{k+1}$. However, by inspecting the $u^{k+1}$ and $v^{k+1}$ updates, we can see that the same terms are involved in the computations except with delayed $x^k$ and $z^k$ updates. Therefore, we can use extra memory to cache the term $Ax^k+Bz^k-c$ (together with $v^k$) at iteration $k$ to be used in the next iteration $u^{k+1}$, which is then fed back to the next $x^{k+1}$ iteration. $H^\top H$, $H^\top b$ as well as $\frac{1}{\lambda_x\lambda}A^\top L$ can also be pre-computed and cached for efficiency. 

Note that the term $(H^\top H x^k -H^\top b)$ is the gradient of the differentiable term, i.e., $\frac{1}{2}\|H x -b\|_2^2$ evaluated at $x^k$. 

Let us now derive a more general form of DFGPGD. Consider a constrained problem
\begin{equation}
 \label{Eq:ProblemDFGPGD}
  \begin{split}
      &\underset{x \in \mathbb{R}^n,z \in \mathbb{R}^n}{\text{minimize}} \,\,\,
      f(x) := g(x) + h(z),\\
      &\text{subject to} \,\,\,
      Ax+Bz = c\,.  
  \end{split}
\end{equation}
where $g$ is convex and differentiable function, $h$ is convex and possibly nondifferentiable, $c \in \mathbb{R}^{p}$, $A \in \mathbb{R}^{p\times n}$ and $B \in \mathbb{R}^{p\times m}$.  

The general form of Dual-Feedback Generalized Proximal Gradient Descent (DFGPGD) algorithm can be written as
\begin{subequations}
\label{iwlm_admm}
    \begin{alignat}{3}
    & u^{k+1} = Ax^k+Bz^k-c+v^k,&\\
    & x^{k+1} = x^k-\frac{1}{\lambda_x}\nabla g(x^k)-\frac{1}{\lambda_x\lambda}A^\top L u^{k+1},&\\
    & z^{k+1} = \underset{z}{\arg\min}\quad h(z)+\frac{1}{2}\|z -\Lambda_{2_k}^{-1}\gamma_{2_k}\|_{\Lambda_{2_k}}^2&\\
    & v^{k+1} = v^k+(Ax^{k+1}+Bz^{k+1}-c),&
    \end{alignat}
which is just an instance of WLM-ADMM~\eqref{WLM-ADMM} with $M_x$ chosen as
\begin{equation}
    M_x = \lambda_xI - \mathbf{H}_g - \frac{1}{\lambda}A^\top L A,
\end{equation}
where $\mathbf{H}_g$ is the hessian\footnote{$\mathbf{H}_g = H^\top H$ when $g(x) = \frac{1}{2}\|Hx-b\|_2^2$.} of $g$ and we require $\lambda_x > \|\mathbf{H}_g\|_2$ for $\Lambda_{1_k}$ to be positive definite (i.e., for the generalized proximal $\operatorname{prox}_{g}$ to be well-defined).
\end{subequations}

\section{Main Results}
\label{Section:Convergence analysis}
Since DFGPGD is an instance of the WLM-ADMM, we can use the same analytical line of proof of~\cite{hamadouche2022probabilistic} to establish convergence guarantees. We also consider here inexact proximal computations with bounded errors and we derive generic deterministic upper bounds. In order to guarantee numerical stability of DFGPGD, we propose a new line of proof that is based on the absolute stability of Lure systems in control. After re-writing the algorithm in state space description and identifying the structural dynamics and feedback nonlinearities, we use Kalman-Yakubovich-Popov lemma to establish sector criteria on the subgradient operators for the numerical stability (we call it \textit{Absolute numerical stability}) and convegence of DFGPGD.  

The main assumptions used during this analysis are stated next.
\subsection{Assumptions}
Let us list the assumptions about problem~\eqref{Eq:LassoProblem0} and proximal errors. Note that some assumptions may only apply in specific cases.
\begin{assumption}[Assumptions on the problem] We consider problem~\eqref{Eq:ProblemDFGPGD} with the following assumptions:
	\label{Assum:Problem}
	\hfill    
	\medskip  
	\noindent
\begin{enumerate}[label=\textbf{P.\arabic*}]
\item The function $g:\mathbb{R}^n \to \mathbb{R} \cup \{+\infty\}$ is differentiable, closed, proper, and convex.

\item The function $h:\mathbb{R}^n \to \mathbb{R} \cup \{+\infty\}$ is
lower-$\mathcal{C}^2$, possibly non-differentiable function.

\item $L \in \mathbb{R}^{p\times p}$, $\Lambda_{1}:=\frac{1}{\lambda} A^\top L  A + M_x \in \mathbb{R}^{n\times n}$ and $\Lambda_{2}:=\frac{1}{\lambda} B^\top L  B + M_z \in \mathbb{R}^{m\times m}$ are positive definite matrices. 
\end{enumerate}
\end{assumption}
\begin{assumption}[Error models] We consider the application of DFGPGD~\eqref{DFGPGD1} to~\eqref{Eq:ProblemDFGPGD} with the following assumptions:
	\label{Assum:ErrModels}
	\hfill    
	\medskip  
	\noindent
\begin{enumerate}[label=\textbf{M.\arabic*}]
\item The error sequences $r_x^k \in \mathbb{R}^n$ and $r_z^k \in \mathbb{R}^n$ in \eqref{DFGPGD1:x-update} and \eqref{DFGPGD1:z-update}, due to proximal errors $\varepsilon_{g}$ and $\varepsilon_{h}$, are deterministic and additive:
\begin{align}
    x^{k} := \overline{x}^{k}+r_x^k  \\
    z^{k} := \overline{z}^{k}+r_z^k, 
\end{align}
where $\overline{x}^{k}$ and $\overline{z}^{k}$ are the exact error-free iterates, that is, the evaluation of~\eqref{DFGPGD1:x-update} and~\eqref{DFGPGD1:x-update} with $\epsilon_g = \epsilon_h = 0$. We define $r_x^0 = r_z^0 = 0$. 
\item The error sequences $r_{x_\Omega}^k \in \mathbb{R}^n$ and $r_{z_\Omega}^k \in \mathbb{R}^n$ in \eqref{DFGPGD1:x-update} and \eqref{DFGPGD1:x-update}, due to random proximal errors $\varepsilon_{g_\Omega}^k$ and $\varepsilon_{h_\Omega}^k$, are probabilistic and additive: 
\begin{align}
    x_{\Omega}^{k} := \overline{x}^{k} + r_{x_\Omega}^k \\
    z_{\Omega}^{k} := \overline{z}^{k} + r_{z_\Omega}^k,
\end{align}
where $\overline{x}^{k}$ and $\overline{z}^{k}$ are the exact error-free iterates. We define $r_{x_\Omega}^0 = r_{z_\Omega}^0 = 0$. The $\Omega$ subscript refers to underlying sample probability space.
\item
For $k\geq 1$, the errors $\epsilon_{g_\Omega}^k$ and $\epsilon_{h_\Omega}^k$ are bounded almost surely. Specifically, for some $\varepsilon_0 > 0$, with probability $1$, we have that
\begin{subequations}
    \begin{alignat}{3}
        &0 \leq \epsilon_{g_\Omega}^k \leq \varepsilon_0&\\
        &0 \leq \epsilon_{h_\Omega}^k \leq \varepsilon_0.&
    \end{alignat}
\end{subequations}  
\item
$\epsilon_{g_\Omega}$ and $\epsilon_{h_\Omega}$ are stationary. Specifically, for all $k$, we have 
\begin{subequations}
    \begin{alignat}{3}
    &\mathbb{E}\bigg[\epsilon_{g_\Omega}^k\bigg] = \mathbb{E}\bigg[\epsilon_{g_\Omega}\bigg]=\text{const.}&\\
    &\mathbb{E}\bigg[\epsilon_{h_\Omega}^k\bigg]=\mathbb{E}\bigg[\epsilon_{g_\Omega}\bigg]=\text{const.}&
    \end{alignat}
\end{subequations}  
\end{enumerate}
\end{assumption}
\subsection{Analytical approach}
Here we use the analytical approach to derive an upper bound on the suboptimality in the function values under proximal errors. The objective is to establish mathematical assertions, in terms of initial conditions of the problem and algorithm's parameters, to test the convergence of DFGPGD proximal computation inaccuracies.
\begin{lemma}
Assume \textbf{P.1}, \textbf{P.2}, \textbf{P.3}, \textbf{M.1} and define the auxiliary sequence $u_2^{k+1} = v^{k+1}+B(z^k-z^{k+1})$. Then for any $x$ and $z$ such that $Ax + Bz = c$,  the sequence generated using the DFGPGD scheme \eqref{DFGPGD1} with  $\lambda_x > \|\mathbf{H}_g\|_2$, where $\mathbf{H}_g$ is the hessian of the differentiable function $g$, and a fixed PSD matrix $L$  satisfies 
\begin{align}
\label{Thm:ConvDFGPGD}
    &\frac{1}{k+1}\sum_{i=0}^{k}f(x^{i+1},z^{i+1})-f(x,z)+\sum_{i=0}^{k}\langle\frac{1}{\lambda}L  u^{i+1},Ax^{i+1}\notag\\&+Bz^{i+1}-(Ax+Bz)\rangle\leq\frac{1}{2(k+1)}\bigg[\lambda_x\|x^0-x\|^2_{}\notag\\&+\|x^{k+1}-x\|^2_{H^\top H}+\frac{1}{\lambda}\|x^{k+1}-x\|^2_{A^\top L A }+\|z^0-z\|^2\bigg]\notag\\
    &+\frac{1}{k+1}\bigg[\sum_{i=0}^{k}\varepsilon_g^{i+1}+\sum_{i=0}^{k}\varepsilon_h^{i+1}
    -\langle  M_x r_x^{k+1},x^{k+1}-x \rangle
    \notag\\&-\langle r_z^{k+1},z^{k+1}-z \rangle\bigg].
\end{align}
\end{lemma}
\begin{theorem}
\label{Theorem1ADMM}
In particular, for a solution $(x,z) = (x^\star,z^\star)$ of problem~\eqref{Eq:ProblemDFGPGD}, we have
\begin{align}
\label{Thm:ConvDFGPGD}
    &\frac{1}{k+1}\sum_{i=0}^{k}f(x^{i+1},z^{i+1})-f(x^\star,z^\star)+\sum_{i=0}^{k}\langle\frac{1}{\lambda}L  u^{i+1},Ax^{i+1}\notag\\&+Bz^{i+1}-(Ax^\star+Bz^\star)\rangle\leq\frac{1}{2(k+1)}\bigg[\lambda_x\|x^0-x^\star\|^2_{}\notag\\&+\|x^{k+1}-x^\star\|^2_{H^\top H}+\frac{1}{\lambda}\|x^{k+1}-x^\star\|^2_{A^\top L A }+\|z^0-z^\star\|^2\bigg]\notag\\
    &+\frac{1}{k+1}\bigg[\sum_{i=0}^{k}\varepsilon_g^{i+1}+\sum_{i=0}^{k}\varepsilon_h^{i+1}
    -\langle  M_x r_x^{k+1},x^{k+1}-x^\star \rangle
    \notag\\&-\langle r_z^{k+1},z^{k+1}-z^\star \rangle\bigg].
\end{align}
\end{theorem}
\begin{proof} 
By applying Lemma~\ref{lemma:Prox2SubGrad} in the appendix to \eqref{DFGPGD1:x-update} and \eqref{DFGPGD1:z-update}, those updates can be rewritten as
\begin{align}
    &\gamma_{1_k}-\Lambda_{1_k}x^{k+1} \in { \partial g}(x^{k+1})\\
    &\gamma_{2_k}-\Lambda_{2_k}z^{k+1} \in  { \partial h}(z^{k+1})\\
    &v^{k+1} = v^k+(Ax^{k+1}+Bz^{k+1}-c).
\end{align}
Using $\varepsilon$-subgradients (cf. Definition~\ref{Def:e-Subgradient} in the appendix), the following inequalities hold for sequences generated by the approximate DFGPGD algorithm \eqref{DFGPGD1}; for any $x$ and $z$,
\begin{align}
\label{Ineq:e-subgrad_g_h}
    &g(x^{k+1})-g(x)\leq\langle\gamma_{1_k}-\Lambda_{1_k}x^{k+1},x^{k+1}-x\rangle+\varepsilon_g^{k+1}\\
    &h(z^{k+1})-h(z)\leq\langle\gamma_{2_k}-\Lambda_{2_k}z^{k+1},z^{k+1}-z\rangle +\varepsilon_h^{k+1}\\
    &v^{k+1} = v^k+(Ax^{k+1}+Bz^{k+1}-c).
\end{align}
Expanding $\gamma_{1_k}$ and $\gamma_{2_k}$ we obtain
\begin{align}
    &g(x^{k+1})-g(x)\leq\langle\Lambda_{1_k} x^k-\frac{1}{\lambda}A^\top L  (Ax^k+Bz^k-c+v^k)\notag\\&-\Lambda_{1_k}x^{k+1},x^{k+1}-x\rangle+\varepsilon_g^{k+1},\\
    &h(z^{k+1})-h(z)\leq\langle\Lambda_{2_k} z^k-\frac{1}{\lambda}B^\top L  (Ax^{k+1}+Bz^k-c+v^k)\notag\\&-\Lambda_{2_k}z^{k+1},z^{k+1}-z\rangle +\varepsilon_h^{k+1},\\
    &v^{k+1} = v^k+(Ax^{k+1}+Bz^{k+1}-c).
\end{align}

\begin{align}
\label{Ineq:e-subgrad_g_h}
    &g(x^{k+1})-g(x)\leq\langle \Lambda_{1_k} (x^k-x^{k+1}),x^{k+1}-x\rangle\notag\\&-\langle\frac{1}{\lambda}A^\top L  u_1^{k+1},x^{k+1}-x\rangle+\varepsilon_g^{k+1},\\
    &h(z^{k+1})-h(z)\leq\langle \Lambda_{2_k} (z^k-z^{k+1}),z^{k+1}-z\rangle\notag\\&-\langle\frac{1}{\lambda}B^\top L  u_2^{k+1},z^{k+1}-z\rangle+\varepsilon_h^{k+1},\\
    &v^{k+1} = v^k+(Ax^{k+1}+Bz^{k+1}-c),
\end{align}
where $u_2^{k+1} = v^{k+1}+B(z^k-z^{k+1})$ and $u_1^{k+1} = v^{k+1}+B(z^k-z^{k+1})+A(x^k-x^{k+1}) = A(x^k-x^{k+1})+u_2^{k+1}$ . Adding both sides of the above inequalities we obtain 
\begin{align}
\label{Ineq:e-subgrad_g_h}
    &f(x^{k+1},z^{k+1})-f(x,z)\leq\langle \Lambda_{1_k} (x^k-x^{k+1}),x^{k+1}-x\rangle\notag\\&+\langle \Lambda_{2_k} (z^k-z^{k+1}),z^{k+1}-z\rangle+\varepsilon_g^{k+1}+\varepsilon_h^{k+1}\notag\\&-\langle\frac{1}{\lambda}L  u_2^{k+1},Ax^{k+1}+Bz^{k+1}-(Ax+Bz)\rangle\notag\\&-\langle\frac{1}{\lambda}A^\top L  A(x^k-x^{k+1}),x^{k+1}-x\rangle.
\end{align}
Substituting $\Lambda_{1_k} = \lambda_xI - H^\top H$ and $\Lambda_{2_k}  = \lambda_z I$ yields
\begin{align}
\label{Ineq:e-subgrad_g_h0}
    &f(x^{k+1},z^{k+1})-f(x,z)\leq\notag\\&\langle (\lambda_xI - H^\top H) (x^k-x^{k+1}),x^{k+1}-x\rangle+\lambda_z\langle z^k-z^{k+1}\notag\\&,z^{k+1}-z\rangle-\langle\frac{1}{\lambda}L  u_2^{k+1},Ax^{k+1}+Bz^{k+1}-(Ax+Bz)\rangle\notag\\
    &+\varepsilon_g^{k+1}+\varepsilon_h^{k+1}-\langle\frac{1}{\lambda}A^\top L  A(x^k-x^{k+1}),x^{k+1}-x\rangle.
\end{align}
For exact iterates $\overline{x}^k$ and $\overline{z}^k$ we have 
\begin{align}
    &\langle M_x (\overline{x}^k-\overline{x}^{k+1}),\overline{x}^{k+1}-x\rangle = \frac{1}{2}\Big[\|\overline{x}^k-x\|_{ M_x }^2-\|\overline{x}^k-\overline{x}^{k+1}\|_{ M_x }^2\notag\\&-\|\overline{x}^{k+1}-x\|_{ M_x }^2\Big] \leq\frac{1}{2}\Big[\|\overline{x}^k-x\|_{ M_x }^2-\|\overline{x}^{k+1}-x\|_{ M_x }^2\Big],
\end{align}
and
\begin{align}
    &\langle \overline{z}^k-\overline{z}^{k+1},\overline{z}^{k+1}-z\rangle = \frac{1}{2}\Big[\|\overline{z}^k-z\|^2-\|\overline{z}^k-\overline{z}^{k+1}\|^2\notag\\&-\|\overline{z}^{k+1}-z\|^2\Big] \leq \frac{1}{2}\Big[\|\overline{z}^k-z\|^2-\|\overline{z}^{k+1}-z\|^2\Big].
\end{align}
Using the above inequalities in \eqref{Ineq:e-subgrad_g_h0} and adding error terms from inexact iterates according to error model \textbf{M.1}, i.e., $x^k=\overline{x}^k+r_x^k$ and $z^k=\overline{z}^k+r_z^k$, we obtain
\begin{align}
\label{Ineq:e-subgrad_g_h}
    &f(x^{k+1},z^{k+1})-f(x,z)+\langle\frac{1}{\lambda}L  u_2^{k+1},Ax^{k+1}+Bz^{k+1}
    \\\notag&-(Ax+Bz)\rangle\leq\frac{\lambda_x}{2}\|x^k-x\|^2_{}-\frac{\lambda_x}{2}\|x^{k+1}-x\|^2_{}
    \\\notag&-\frac{1}{2}\|x^k-x\|^2_{ H^\top H }+\frac{1}{2}\|x^{k+1}-x\|^2_{ H^\top H }-\frac{1}{2\lambda}\|x^k-x\|^2_{A^\top L A }
    \\\notag&+\frac{1}{2\lambda}\|x^{k+1}-x\|^2_{A^\top L A }+\frac{1}{2}\|z^k-z\|^2-\frac{1}{2}\|z^{k+1}-z\|^2+\varepsilon_g^{k+1}
    \\\notag&+\varepsilon_h^{k+1}+\langle M_x  r_x^k,x^k-x \rangle+\langle r_z^k,z^k-z \rangle
    \\\notag&
    -\langle  M_x r_x^{k+1},x^{k+1}-x \rangle-\langle r_z^{k+1},z^{k+1}-z \rangle+\frac{1}{2}\big(\| M_x  r_x^{k}\|^2\\\notag&+\|r_z^{k}\|^2-\| M_x  r_x^{k+1}\|^2-\|r_z^{k+1}\|^2\big).
\end{align}
where $M_x = \lambda_xI - H^\top H - \frac{1}{\lambda}A^\top L A$. Summing both sides of the inequality from $0$ to $k$ yields
\begin{align}
\label{Ineq:e-subgrad_g_h}
    &\sum_{i=0}^{k}f(x^{i+1},z^{i+1})-(k+1)f(x,z)+\sum_{i=0}^{k}\langle\frac{1}{\lambda}L  u_2^{i+1},Ax^{i+1}\notag\\&+Bz^{i+1}-(Ax+Bz)\rangle\leq\notag
    \frac{\lambda_x}{2}\|x^0-x\|^2_{}-\frac{\lambda_x}{2}\|x^{k+1}-x\|^2_{}\notag\\&-\frac{1}{2}\|x^0-x\|^2_{H^\top H}+\frac{1}{2}\|x^{k+1}-x\|^2_{H^\top H}
    -\frac{1}{2\lambda}\|x^0-x\|^2_{A^\top L A }\notag\\&+\frac{1}{2\lambda}\|x^{k+1}-x\|^2_{A^\top L A }+\frac{1}{2}\|z^0-z\|^2-\frac{1}{2}\|z^{k+1}-z\|^2\notag\\&+\sum_{i=0}^{k}\varepsilon_g^{i+1}+\sum_{i=0}^{k}\varepsilon_h^{i+1}-\langle r_z^{k+1},z^{k+1}-z \rangle\notag\\
    &+\frac{1}{2}\big(\| M_x  r_x^{0}\|^2+\|r_z^{0}\|^2-\| M_x  r_x^{k+1}\|^2-\|r_z^{k+1}\|^2\big)
    \notag\\&+\langle M_x  r_x^0,x^0-x \rangle+\langle r_z^0,z^0-z \rangle-\langle  M_x r_x^{k+1},x^{k+1}-x \rangle.
\end{align}
Using the initial conditions of error model \textbf{M.1}; i.e., $r_x^0 = r_z^0 = 0$, we obtain 
\begin{align}
\label{Ineq:e-subgrad_g_h}
    &\sum_{i=0}^{k}f(x^{i+1},z^{i+1})-(k+1)f(x,z)\notag\\&+\sum_{i=0}^{k}\langle\frac{1}{\lambda}L  u_2^{i+1},Ax^{i+1}+Bz^{i+1}-(Ax+Bz)\rangle\leq\notag
    \\
    &\frac{\lambda_x}{2}\|x^0-x\|^2_{}-\frac{\lambda_x}{2}\|x^{k+1}-x\|^2_{}-\frac{1}{2}\|x^0-x\|^2_{H^\top H}\notag\\&+\frac{1}{2}\|x^{k+1}-x\|^2_{H^\top H}-\frac{1}{2\lambda}\|x^0-x\|^2_{A^\top L A }\notag\\&+\frac{1}{2\lambda}\|x^{k+1}-x\|^2_{A^\top L A }+\frac{1}{2}\|z^0-z\|^2\notag\\&-\frac{1}{2}\|z^{k+1}-z\|^2+\sum_{i=0}^{k}\varepsilon_g^{i+1}+\sum_{i=0}^{k}\varepsilon_h^{i+1}\notag\\&-\| M_x  r_x^{k+1}\|^2-\|r_z^{k+1}\|^2\big)
    -\langle  M_x r_x^{k+1},x^{k+1}-x \rangle
    \notag\\&-\langle r_z^{k+1},z^{k+1}-z \rangle.
\end{align}
Dividing both sides by $k+1$ completes the proof.
\begin{align}
\label{Ineq:e-subgrad_g_h}
    &\frac{1}{k+1}\sum_{i=0}^{k}f(x^{i+1},z^{i+1})-f(x,z)
    \notag\\&+\sum_{i=0}^{k}\langle\frac{1}{\lambda}L  u_2^{i+1},Ax^{i+1}+Bz^{i+1}-(Ax+Bz)\rangle\leq\notag
    \\
    &\frac{1}{2(k+1)}\bigg[\lambda_x\|x^0-x\|^2_{}+\|x^{k+1}-x\|^2_{H^\top H}
    \notag\\&+\frac{1}{\lambda}\|x^{k+1}-x\|^2_{A^\top L A }+\|z^0-z\|^2\bigg]+\frac{1}{k+1}\bigg[\sum_{i=0}^{k}\varepsilon_g^{i+1}\notag\\
    &+\sum_{i=0}^{k}\varepsilon_h^{i+1}-\langle  M_x r_x^{k+1},x^{k+1}-x \rangle-\langle r_z^{k+1},z^{k+1}-z \rangle\bigg].
\end{align}
where we have dropped negative terms from the right hand side.
\end{proof}
\subsection{Dynamical system approach}
In order to prove the convergence of DFGPGD, let us derive its dynamical state space description. Recall that for a generalized proximal WLM-ADMM, the state space description is given by the following differential inclusion

\begin{align}
&\begin{bmatrix} \Lambda_{1_k} & 0 & 0 \\  0 & \Lambda_{2_k} & 0 \\ 0 & 0 & I \end{bmatrix}
   \begin{bmatrix}\dot{x}(t)\\ \dot{z}(t)\\ \dot{y}(t) \end{bmatrix} \in \begin{bmatrix}0 & 0 & -A^\top L_k\\ 0 & 0 & -B^\top L_k\\  L_kA & L_kB& 0\end{bmatrix}  \begin{bmatrix}x(t)\\z(t)\\ y(t)\end{bmatrix}\notag\\& - \begin{bmatrix} \partial g(x(t)) \\\partial h(z(t)) \\ c \end{bmatrix}.
\end{align}
where the feedback nonlinearity is given by $\Phi = \begin{bmatrix} \partial g \\\partial h \\ 0 \end{bmatrix}$.

For DFGPGD, substituting $\Lambda_{1_k} = \lambda_x I - \mathbf{H}_g$ and fixing $L_k = L$, we obtain the following
\begin{align}
&\begin{bmatrix} \lambda_x I - \mathbf{H}_g & 0 & 0 \\  0 & \Lambda_{2_k} & 0 \\ 0 & 0 & I \end{bmatrix}
   \begin{bmatrix}\dot{x}(t)\\ \dot{z}(t)\\ \dot{y}(t) \end{bmatrix} \in \notag\\& \begin{bmatrix}0 & 0 & -A^\top L\\ 0 & 0 & -B^\top L\\  LA & LB& 0\end{bmatrix}  \begin{bmatrix}x(t)\\z(t)\\ y(t)\end{bmatrix} - \begin{bmatrix} \partial g(x(t)) \\\partial h(z(t)) \\ c \end{bmatrix}.
   \label{Eq:DFGPGDLureDescriptor}
\end{align}

In order for $\Phi$ to satisfy the sector condition we require $g,h : \mathbb{R}^n \rightarrow \mathbb{R} \cup \{+\infty\}$ to be convex lower semi-continuous, so that $\partial g$ and $\partial h$ are maximally monotone multivalued mappings.

Define 
\begin{align}
    &E = \begin{bmatrix} \lambda_x I - \mathbf{H}_g & 0 & 0 \\  0 & \Lambda_{2_k} & 0 \\ 0 & 0 & I \end{bmatrix}, \Gamma = \begin{bmatrix}0 & 0 & -A^\top L\\ 0 & 0 & -B^\top L\\  LA & LB& 0\end{bmatrix},\notag\\& \Sigma = I
\end{align}

Assume the following:
\begin{enumerate}[label=\textbf{A.5.\arabic*}]
\item The open-loop linear system $(E,\Gamma,\Sigma)$ (i.e., without feedback nonlinearity) is minimal, i.e. controllable and observable 
\begin{equation}
 \operatorname{rank}\begin{bmatrix}E^{-1}\Sigma & E^{-1}\Gamma E^{-1}\Sigma & \dots & (E^{-1}\Gamma)^{n-1}E^{-1}\Sigma\end{bmatrix} = n,
\end{equation}
and
\begin{equation}
 \operatorname{rank}\begin{bmatrix}\Omega \\ \Omega E^{-1}\Gamma \\ \vdots \\ \Omega (E^{-1}\Gamma)^{n-1}\end{bmatrix} = n.
\end{equation}
\end{enumerate}

Then for the transfer function matrix $G(s) = \Omega(sE - \Gamma)^{-1}\Sigma$ to be strictly positive real (SPR) we need to solve the following algebraic problem. According to the generalized Kalman-Yakubovich-Popov (KYP) Lemma, we need to find positive definite matrices $P = P^\top$ and $Q = Q^\top$ such that 
\begin{equation}    
\begin{split}
&P E^{-1}\Gamma + \Gamma^\top (E^{-1})^\top P = -Q \\ 
&\Sigma^\top (E^{-1})^\top P = \Omega.
\end{split}
\end{equation}
Once we obtain $P$ and $Q$, then for absolute numerical stability of DFGPGD we need to estimate the PD matrices $\underline{K}$ and $\overline{K}$ such that:
\begin{itemize}
    \item $E^{-1}\Phi$ satisfies the sector condition, i.e., $
    (E^{-1}\Phi(\zeta)-\underline{K}\zeta)^\top(E^{-1}\Phi(\zeta)-\overline{K}\zeta) \leq  0, \forall \zeta \in \mathbb{R}^{3n}$; and
    \item $\bigg(I + \overline{K}G(s)\bigg)\bigg(I + \underline{K}G(s)\bigg)^{-1}$ is strictly positive real.
\end{itemize}

\begin{theorem}
Consider the candidate Lyapunov function $V (\zeta) = \frac{1}{2} (\zeta-\zeta^\star)^\top P (\zeta-\zeta^\star)$ and assume the following:
\begin{enumerate}[label=\textbf{A.2.\arabic*}]
\item The linear system (without feedback nonlinearity) is assumed to be minimal, i.e. controllable and observable which means that
\begin{equation}
 \operatorname{rank}\begin{bmatrix}E^{-1}\Sigma & E^{-1}\Gamma E^{-1}\Sigma & \dots & (E^{-1}\Gamma)^{n-1}E^{-1}\Sigma\end{bmatrix} = n,
\end{equation}
and
\begin{equation}
 \operatorname{rank}\begin{bmatrix}\Omega \\ \Omega E^{-1}\Gamma \\ \vdots \\ \Omega (E^{-1}\Gamma)^{n-1}\end{bmatrix} = n.
\end{equation}
\item $G(s) = \Omega(sE - \Gamma)^{-1}\Sigma$, with $(E,\Gamma, \Sigma, \Omega)$ a minimal representation, is a SPR transfer matrix. Or according to the generalized Kalman-Yakubovich-Popov (KYP) Lemma, there exists positive definite matrices $P = P^\top$ and $Q = Q^\top$ such that $P E^{-1}\Gamma + \Gamma^\top (E^{-1})^\top P = -Q$ and $\Sigma^\top (E^{-1})^\top P = \Omega$.

\item  $E^{-1}\Sigma$ is full column rank, equivalently $\mathcal{K}er(E^{-1}\Sigma)$ = \{0\}. 

\item  $g : \mathbb{R}^n \rightarrow \mathbb{R} \cup \{+\infty\}$ is convex lower semi-continuous, so that $\partial g$ is a maximally monotone multivalued mapping.

\item  $h : \mathbb{R}^n \rightarrow \mathbb{R} \cup \{+\infty\}$ is convex lower semi-continuous, so that $\partial h$ is a maximally monotone multivalued mapping.

\item $\Phi$ satisfies the sector condition, i.e., there are PD matrices $\underline{K}$ and $\overline{K}$ such that: $
(E^{-1}\Phi(\zeta)-\underline{K}\zeta)^\top(E^{-1}\Phi(\zeta)-\overline{K}\zeta) \leq  0 \quad, \forall \zeta \in \mathbb{R}^{3n}$

\item $\bigg(I + \overline{K}G(s)\bigg)\bigg(I + \underline{K}G(s)\bigg)^{-1}$ is strictly positive real.

\item $\Omega E \zeta(0) \in \operatorname{dom} \partial \Phi$, 
\end{enumerate}

then the DFGPGD algorithm is numerically absolutely exponentially stable and $\zeta(t)$ satisfies
\begin{equation}
\frac{1}{2}\|\zeta-\zeta^\star\|_2^2 \leq \frac{1}{\lambda_{\min}(P)}V(0) \exp{\bigg(-2\frac{\lambda_{\min}(Q)}{\lambda_{\min}(P)}t\bigg)}.
\end{equation}
\end{theorem}

\section{Experimental results}
\label{Section:Experimental results}
We have implemented ADMM and DFGPGD using a C++ synthesizable linear algebra library (SXLAL\footnote{https://github.com/wincle626/SXLAL}) and Fixed-Point arithmetic from Xilinx' arbitrary precision library \footnote{https://docs.xilinx.com/r/en-US/ug1399-vitis-hls/Arbitrary-Precision-Data-Types-Library}. We set the wordlength (W) to $24$ bits and we allocate $16$ bits for the integer part and $8$ bits to the fractional part. We run the Vivado HLS compiler to generate the hardware block (IP core) and obtain the timing and resource utilization reports, which are summarized in Table~\ref{table:1} and Table~\ref{table:2}, respectively. 

The timing report in Table~\ref{table:1} shows the hardware-based algorithms latency in terms of estimated clock cycles for a given target of $10ns$ clock period. The estimated and achieved clock period (CP) can be read from the post-synthesis and post-implementation reports. In our case, ADMM  achieves a clock period of $6.372ns$ and DFGPGD achieves $4.097ns$, which is $35.70\%$ lower than ADMM (or $55\%$ higher clock frequency). The timing constraints are met in both designs.

In this experiment, we used 5 iterations in both solvers for comparison, this is referred to as  trip count in Vivado HLS timing reports. The iteration latency is reduced from $1401$ cycles\footnote{Note that in order to estimate the latency for ADMM, we have used a more optimized ADMM implementation which makes use of pre-cached Cholesky inverse of $H^\top H$, otherwise the synthesis reports an unknown value.} in ADMM to only $1052$ cycles using DFGPGD, which is $23.39\%$ reduction in overall latency (i.e., from $7329$ cycles to $5615$ cycles as shown in Table~\ref{table:1} below).
\begin{table}[!h]
\centering
\caption{Timing report}
\label{table:1}
\makebox[0.5 \textwidth][c]{       
\resizebox{0.5 \textwidth}{!}{   
\begin{tabular}{||r|ccc||} 
 \cline{2-4}
 
 \multicolumn{1}{r|}{} & Latency (cycles) & CP post-synth. (ns) & CP post-impl. (ns) \\ 
 \hline
\textbf{ADMM} &	$7329$ &  $4.305$ & $6.372$ \\
\textbf{DFGPGD} &	$5615$ & $3.195$ & $4.097$ \\
\textbf{Diff.} & \textbf{-23.39\%}   & \textbf{-25.78\%} &	\textbf{-35.70\%} \\ 
 \hline
\end{tabular}
}
}
\end{table}

The  resource utilization report is summarized in Table~\ref{table:2} below. DSP48E stands for digital signal processing logic element, FF for Flip Flops and LUT are Look-Up-Table blocks in FPGA. From Table~\ref{table:2}, we can see that DFGPGD uses $44.6\%$ less Flip Flops and it greatly reduces the digital signal processing logic and LUT block utilizaiton by $37.50\%$ and $38.53\%$, respectively. 
\begin{table}[!h]
\centering
\caption{Resource utilization report from post-implementation}
\label{table:2}
\makebox[0.5 \textwidth][c]{       
\resizebox{0.5 \textwidth}{!}{   
\begin{tabular}{||r|cccc||} 
 \cline{2-5}
 
 \multicolumn{1}{r|}{}& CLB & DSP48E & FF & LUT  \\ 
 \hline
\textbf{ADMM} & 339 & 16 & 1577 & 1822 \\
\textbf{DFGPGD}& 204 & 10	& 879	& 1120 \\
\textbf{Diff.} &\textbf{-39.82\%}	&\textbf{-37.50\%}	&\textbf{-44.26\%} & \textbf{-38.53\%}\\ 
 \hline
\end{tabular}
}
}
\end{table}

\begin{table}[!h]
\centering
\caption{Power report from post-implementation}
\label{table:2}
\makebox[0.5 \textwidth][c]{       
\resizebox{0.5 \textwidth}{!}{   
\begin{tabular}{||r|cc||} 
 \cline{2-3}
 
 \multicolumn{1}{r|}{}& Total Power (W) & Dynamic Power (W) \\ 
 \hline
\textbf{ADMM}& 0.636 & 0.044 \\
\textbf{DFGPGD}& 0.611 & 0.019 \\
\textbf{Diff.} & \textbf{-3.93\%} & \textbf{-56.82\%}\\ 
 \hline
\end{tabular}
}
}
\end{table}

The combined reduction in resource utilization and latency results in a $56.82\%$ reduction in dynamic power which accounts for a $3.93\%$ decrease in total on-chip power as reported in Table~\ref{table:2}. 

The dynamic power dissipation per iteration is approximately $8.8\text{E-}3$ W for ADMM and $3.8\text{E-}3$ W for DFGPGD. However, ADMM compensates for this loss by requiring a less number of iterations for a given error tolerance. Let the relative error for a synthetic LASSO problem be given by: $(f-f^\star)/f \%$, the dynamic power dissipation as a function of relative error is shown below. We run 151 experiments with randomly generated data that automatically satisfy the following  reconstructibility condition
\begin{equation}
     m > 2\cdot s\cdot\log\bigg(\frac{n}{s}\bigg)+\frac{7}{5}\cdot s+1.
\end{equation}
We control the relative error in every experiment via the maximum number of iterations (from $\operatorname{MAX\_ITER}=10$ to $\operatorname{MAX\_ITER}=300$) of the main loop in both algorithms. 

\begin{figure}[!htb]
\includegraphics[width=9cm]{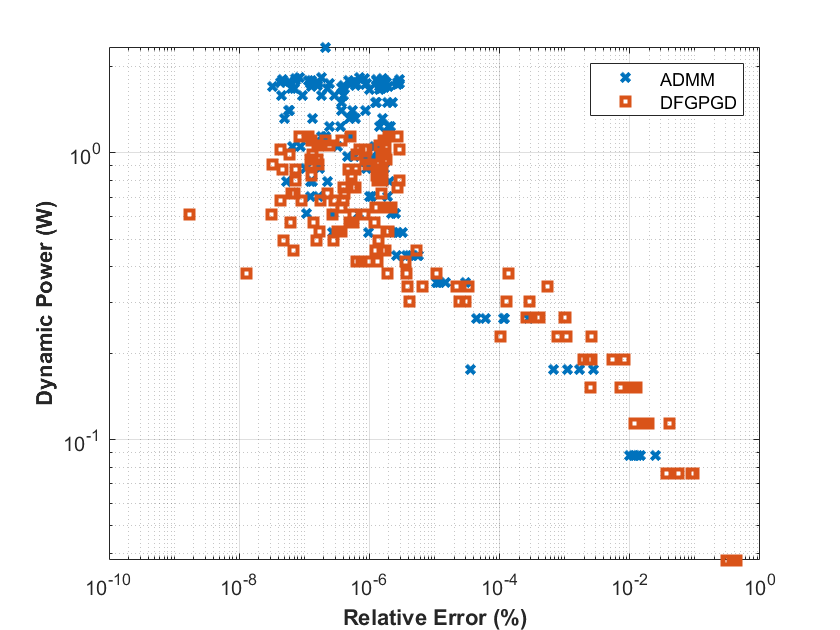}
\caption{Dynamic Power vs Relative Error ($|f-f^\star|/f^\star$) for fixed $n = 700$ and $m=270$ for DFGPGD and ADMM.}
\label{Figure5.1}
\end{figure}

From~\autoref{Figure5.1}, we can see that ADMM dissipates more dynamic power at high precision but quickly shifts to low power mode for a slight performance degradation. DFGPGD is more power-efficient but the power gain (efficiency) in terms of reduced accuracy (performance) is negligible as demonstrated by the slow transition to low power modes.




\section{Conclusions}
\label{Section:Conclusions}
In this work, we analysed the convergence of a new operator splitting algorithm which we called "Dual-Feedback Generalized Proximal Gradient Descent (DFGPGD)" and we established numerical stability guarantees in terms of sector conditions on the subdifferential of the objective function.  We implemented DFGPGD and ADMM on FPGA ZCU106 board and we compared them with respect to clock frequency, latency as well as resource utilization and power efficiency. We also performed a full-stack, application-to-hardware, comparison between approximate versions of DFGPGD and ADMM based on power/error rate trade-off.
The new generalized proximal ADMM (WLM-ADMM) can be used to instantiate new composite optimization solvers. As a future work, we will exploit the degrees of freedom of WLM-ADMM to design new custom, hardware-friendly algorithms that target more realistic applications. 
\addcontentsline{toc}{section}{Acknowledgment}

\appendix

\begin{lemma}
\label{lemma:Prox2SubGrad}
Let $x \in \mathbb{R}^n$. Then $x^{+} = \text{prox}_{q}(x)$ if, and only if, $(x-x^{+})\in \partial q(x^{+})$.
\end{lemma}

\begin{definition}[$\varepsilon$-subgradient]
\label{Def:e-Subgradient}
Let the $\varepsilon$-subgradient $w \in \partial_{\varepsilon} f(x)$, then $\forall y \in \mathcal{H}$ we have
\begin{equation}
    f(y) \geq f(x) + \langle w , y-x \rangle-\varepsilon.
\end{equation}
\end{definition}

\printbibliography
\end{document}